\documentclass[reqno]{amsart}
\usepackage[utf8x]{inputenc}  

\usepackage{setspace}
\usepackage[left=2.8cm, right=2.8cm, bottom=2.5cm]{geometry}                 

\usepackage{graphicx} 
\usepackage{pgf,tikz}
\usetikzlibrary{arrows}
\usepackage{amssymb}
\usepackage{amsmath}
\usepackage{pdfsync}
\usepackage{mathrsfs}
\usepackage{hyperref} 
\usepackage{verbatim} 
\usepackage{epstopdf}
\usepackage{bbm} 
\usepackage[colorinlistoftodos,prependcaption,textsize=tiny]{todonotes}
\usepackage{xargs}
\usepackage{lmodern}

\def\R{\mathbb{R}}

\def\Z{\mathbb{Z}}
\def\C{\mathbb{C}}

 \newcommand{\dt}{\text{\rm d}t}

 \newcommand{\dx}{\text{\rm d}x}

\newcommandx{\emanuel}[2][1=]{\todo[linecolor=green,backgroundcolor=green!25,bordercolor=black,#1]{#2}}

\newcommandx{\diogo}[2][1=]{\todo[linecolor=orange,backgroundcolor=orange!25,bordercolor=orange,#1]{#2}}

\newcommandx{\mateus}[2][1=]{\todo[linecolor=blue,backgroundcolor=blue!25,bordercolor=blue,#1]{#2}}

\newcommandx{\danger}[2][1=]{\todo[linecolor=red,backgroundcolor=red!25,bordercolor=blue,#1]{#2}}

\renewcommand{\d}{\text{\rm d}}

\newtheorem{theorem}{Theorem}

\newtheorem{proposition}[theorem]{Proposition}
\newtheorem{lemma}[theorem]{Lemma}

\makeatletter
\DeclareFontFamily{U}{tipa}{}
\DeclareFontShape{U}{tipa}{m}{n}{<->tipa10}{}
\newcommand{\arc@char}{{\usefont{U}{tipa}{m}{n}\symbol{62}}}%

\newcommand{\arc}[1]{\mathpalette\arc@arc{#1}}

\newcommand{\arc@arc}[2]{%
  \sbox0{$\m@th#1#2$}%
  \vbox{
    \hbox{\resizebox{\wd0}{\height}{\arc@char}}
    \nointerlineskip
    \box0
  }%
}
\makeatother

\numberwithin{equation}{section}

\allowdisplaybreaks

\newcommand{\intav}[1]{\mathchoice {\mathop{\vrule width 6pt height 3 pt depth  -2.5pt
\kern -8pt \intop}\nolimits_{\kern -6pt#1}} {\mathop{\vrule width
5pt height 3  pt depth -2.6pt \kern -6pt \intop}\nolimits_{#1}}
{\mathop{\vrule width 5pt height 3 pt depth -2.6pt \kern -6pt
\intop}\nolimits_{#1}} {\mathop{\vrule width 5pt height 3 pt depth
-2.6pt \kern -6pt \intop}\nolimits_{#1}}}

\newcommand{\intavl}[1]{\mathchoice {\mathop{\vrule width 6pt height 3 pt depth  -2.5pt
\kern -8pt \intop}\limits_{\kern -6pt#1}} {\mathop{\vrule width 5pt
height 3  pt depth -2.6pt \kern -6pt \intop}\nolimits_{#1}}
{\mathop{\vrule width 5pt height 3 pt depth -2.6pt \kern -6pt
\intop}\nolimits_{#1}} {\mathop{\vrule width 5pt height 3 pt depth
-2.6pt \kern -6pt \intop}\nolimits_{#1}}}

\title[Hilbert transforms and the equidistribution of zeros of polynomials]{Hilbert transforms and the \\ equidistribution of zeros of polynomials}

\author[AIM Group Project]{Emanuel Carneiro, Mithun Kumar Das, Alexandra Florea, Angel V. Kumchev,\\ Amita Malik, Micah B. Milinovich, Caroline Turnage-Butterbaugh, and Jiuya Wang}

\address{
ICTP - The Abdus Salam International Centre for Theoretical Physics\\
Strada Costiera, 11, I - 34151, Trieste, Italy \  and \  IMPA - Instituto de Matem\'{a}tica Pura e Aplicada\\
Rio de Janeiro - RJ, Brazil, 22460-320.}

\email{carneiro@ictp.it}
\email{carneiro@impa.br}

\address{Indian Institute of Science Education and Research Berhampur, Engg.~School Road, Berhampur, India-760010.}
\email{das.mithun3@gmail.com}

\address{University of California, Irvine, Department of Mathematics, Irvine, CA 92697.}
\email{floreaa@uci.edu}

\address{Department of Mathematics, Towson University, 8000 York Road, Towson, MD 21252, USA.}
\email{akumchev@towson.edu}

\address{American Institute of Mathematics,
600 East Brokaw Road,
San Jose, CA 95112-1006, USA \  and \  Max-Planck-Institut f\"{u}r Mathematik, Vivatsgasse 7, 53111 Bonn, Germany 
.}

\email{amita.malik@aimath.org}
\email{malik@mpim-bonn.mpg.de}

\address{Department of Mathematics, University of Mississippi, University, MS 38677, USA.}
\email{mbmilino@olemiss.edu}

\address{Department of Mathematics and Statistics, Carleton College, Northfield, MN 55057, USA.}
\email{cturnageb@carleton.edu}

\address{Department of Mathematics, Duke University, 120 Science Dr, Durham, NC, 27708, USA.}
\email{wangjiuy@math.duke.edu}

\date{\today}                                           
\begin{document}

\begin{abstract} 
We improve the current bounds for an inequality of Erd\H{o}s and Tur\'an from 1950 related to the discrepancy of angular equidistribution of the zeros of a given polynomial. Building upon a recent work of Soundararajan, we establish a novel connection between this inequality and an extremal problem in Fourier analysis involving the maxima of Hilbert transforms, for which we provide a complete solution. Prior to Soundararajan (2019), refinements of the discrepancy inequality of Erd\H{o}s and Tur\'an had been obtained by Ganelius (1954) and Mignotte (1992).

\end{abstract}

\maketitle

\section{Introduction}

\subsection{Background} 
Following the elegant treatment of Soundararajan \cite{S}, we revisit the classical work of Erd\H{o}s and Tur\'an \cite{ET} on the distribution of zeros of polynomials in the complex plane. In particular, we establish a connection between the upper bound for the discrepancy of the angles of the zeros of a given polynomial and an extremal problem in Fourier analysis involving the maxima of Hilbert transforms. Before describing this extremal problem, which is solved completely in this paper, we first describe our application in number theory. 

Let
\[
P(z) = \prod_{j=1}^N \big( z-\alpha_j\big) = z^N + a_{N-1} z^{N-1} + \cdots + a_0
\]
be a monic polynomial of degree $N$, with $a_0 \neq 0$ and roots $\alpha_j=\rho_j \, e^{2\pi i\theta_j}$. Roughly speaking, 
Erd\H{o}s and Tur\'an proved that if the {\it size} of $P(z)$ on the unit circle is small,
and $a_0$ is not too small, then its roots cluster around the unit circle and
the angles $2\pi\theta_j$ become equidistributed as $N \to \infty$. Two notions of size, or {\it height}, of a polynomial that have been considered in this problem are
$$H(P) = \max_{|z| =1} \, \frac{|P(z)|}{\sqrt{|a_0|}} \ \ \ {\rm and} \ \ \ h(P)  =  \int_{0}^{1} \log^+\!\left(\frac{\big|P\big(e^{2\pi i \theta}\big)\big|}{\sqrt{|a_0|}} \right)\,\d\theta,$$
where  $\log^+\!x = \max\{\log x, 0\}$. By Parseval's identity, we have 
\[
\int_{0}^1 |P\big(e^{2\pi i \theta}\big)\big|^2\,\d\theta = 1 + |a_{N-1}|^2 + \ldots + |a_0|^2,
\] 
from which it follows easily that  $H(P) \geq 1$ and therefore $h(P) \leq \log H(P)$. Hence, the assumption that $h(P)$ is small is weaker than the assumption that $H(P)$ is small. Let us also define the quantity 
$$\mathcal{M}(P) = \prod_{j=1}^N\max \left\{ \rho_j, \frac{1}{\rho_j}\right\}.$$
The observation that the zeros cluster around the unit circle is given by the inequality \cite[Theorem 1]{S}
$$\log \mathcal{M}(P) \leq 2\, h(P),$$
that follows by an interesting application of Jensen's formula in complex analysis. \\

We focus on the study of the equidistribution of the angles $2\pi\theta_j$. Given an interval $I$ on $\R/\Z$, we let $N(I;P)$ denote the number of zeros $\alpha_j = \rho_j \, e^{2\pi i \theta_j}$ for which $\theta_j \in I$. A convenient way to measure the distribution of the sequence $\{\theta_j\}_{j=1}^N$ is by means of its {\it discrepancy}, defined by
$$\mathcal{D}(P):= \sup_{I} \Big| N(I; P) - |I| N\Big|,$$
where $|I|$ denotes the length of the interval $I$. We list a few notable results in estimating the discrepancy $\mathcal{D}(P)$. Erd\H{o}s and Tur\'an, in their original paper \cite{ET} of 1950, proved that
\begin{equation}\label{20210227_14:46}
\mathcal{D}(P) \le C \sqrt{ N \log H(P)},
\end{equation}
with $C = 16$. In 1954, Ganelius \cite{G} established \eqref{20210227_14:46} with the constant $C = \sqrt{2\pi / k} =2.5619\ldots$, where $k=1/1^2 - 1/3^2 + 1/5^2 - \ldots =0.9159\ldots$ denotes Catalan's constant. Amoroso and Mignotte \cite{AM} have produced examples that show that the constant $C$ in \eqref{20210227_14:46} must be at least $\sqrt{2}$. In 1992, Mignotte \cite{M} refined Ganelius's result by establishing the stronger inequality
\begin{equation}\label{20210227_14:47}
\mathcal{D}(P) \le C \sqrt{ N \,h(P)},
\end{equation}
with the same constant $C = \sqrt{2\pi / k} =2.5619\ldots$. Only recently, in 2019, Soundararajan \cite{S} improved this result by establishing \eqref{20210227_14:47} with the constant 
$$C = \frac{8}{\pi} = 2.5464\ldots.$$

Our goal is to provide an improvement of the admissible value of $C$ in \eqref{20210227_14:47}. We follow the general outline of proof of Soundararajan in \cite{S} up to a certain point, then we diverge and introduce a novel ingredient:~the connection to a certain extremal problem in Fourier analysis involving the maxima of Hilbert transforms.  As a direct consequence of Theorems \ref{Thm1} and \ref{Thm2} below, we prove that the constant
\[
C = \frac{\ \ 4}{\sqrt{\pi}} = 2.2567\ldots
\]
is admissible in \eqref{20210227_14:47} and show that this constant is the best possible with our particular strategy.

\begin{theorem}\label{Cor2} If $P$ is a monic polynomial of degree $N$ with $P(0)\ne 0$, then 
$$\mathcal{D}(P) \le \frac{\ \ 4}{\sqrt{\pi}}  \, \sqrt{ N \,h(P)}.$$
\end{theorem}

{\sc Remark}. In their original paper \cite{ET}, Erd\H{o}s and Tur\'an were also interested in estimating the number $\mathcal{R}(P)$ of real roots of a polynomial $P$. In particular, the notion of discrepancy can be used towards this goal. From the definition, letting $I$ denote either the point $0$ or $\frac{1}{2}$, it plainly follows that $\mathcal{R}(P) \le 2 \, \mathcal{D}(P)$.

We also note that, for the example $P(z) = (z-1)^N$, one has $\mathcal{D}(P) = N$ and 
\begin{equation}\label{20210302_10:20}
h(P) =  N \int_{0}^{1} \log^+\!\big|e^{2\pi i \theta} - 1\big| \,\d\theta = N \, \frac{ 3 \sqrt{3} \, L(2, \chi_3)}{4\pi},
\end{equation}
where $\chi_3$ denotes the quadratic character modulo $3$. This last identity was observed by C.~J.~Smyth in a slightly different context, see \cite[Appendix 1]{B}. Hence, the constant $C$ in \eqref{20210227_14:47} cannot be smaller than 
\begin{equation}\label{20210302_21:30}
\sqrt{\frac{4\pi}{3 \sqrt{3} \, L(2, \chi_3)}} = 1.75936\ldots.
\end{equation}

\smallskip

\subsection{Fourier optimization} Throughout this paper we consider functions in two different environments: the ones defined on $\R$ (usually denoted here with capital letters) and the ones defined on $\R/\Z$ (usually denoted here with lower case letters). 

\smallskip

For $F \in L^1(\R)$ we define its Fourier transform $\widehat{F}:\R \to \C$ by 
$$\widehat{F}(t) = \int_{-\infty}^{\infty} e^{- 2\pi i t x}\,F(x)\,\d x.$$
By Plancherel's theorem one can extend the Fourier transform to an isometry on $L^2(\R)$.  The Hilbert transform $\mathcal{H}$ is another classical operator in harmonic analysis that has a few (equivalent) interpretations. As a singular integral it is defined by 
 \begin{equation}\label{20210316_12:09}
 \mathcal{H}(F)(x) = {\rm p.v.}\,\frac{1}{\pi}  \int_{\R} F(x - t) \,\frac{1}{t}\,\d t\,,
 \end{equation}
where the notation p.v.~here means that such integral should be understood as a Cauchy principal value. The classical theory of singular integrals guarantees that the Hilbert transform is a well-defined operator on $L^p(\R)$ for $1 \leq p < \infty$, being a bounded operator if $1 < p < \infty$, and satisfying a weak-type-$(1,1)$ estimate when $p=1$. See, for instance \cite[Chapters V and VI]{SW} or \cite[Chapter 4]{Graf_book} for proofs of these facts and the connections with the theory of conjugate harmonic functions. In particular, the appropriate limiting process in \eqref{20210316_12:09} converges a.e.~for $F \in L^p(\R)$, $1 \leq p < \infty$. The operator $\mathcal{H}: L^2(\R) \to L^2(\R)$ is an isometry that can be alternatively defined on the Fourier space by the relation\footnote{Recall that ${\rm sgn}:\R \to \R$ is defined by ${\rm sgn}(t) =1$, if $t>0$; ${\rm sgn}(0) = 0$; and ${\rm sgn}(t) =-1$, if $t<0$.} 
\begin{equation}\label{20210301_09:24}
\widehat{\mathcal{H}(F)}(t) = -i \, {\rm sgn}(t) \, \widehat{F}(t).
\end{equation}

Similarly, in the periodic setting, if $f \in L^1(\R/\Z)$ we define its Fourier transform $\widehat{f}:\Z \to \C$ by 
$$\widehat{f}(k) =  \int_{\R/\Z} e^{- 2\pi i k \theta}\,f(\theta)\,\d \theta.$$
The periodic Hilbert transform is the singular integral operator defined by
\begin{equation}\label{20210316_12:14}
\mathcal{H}(f)(\theta) = {\rm p.v.}\, \int_{\R/\Z} f(\theta - \alpha) \,\cot(\pi \alpha)\,\d \alpha.
\end{equation}
Again, the appropriate limiting process in \eqref{20210316_12:14} converges a.e.~if $f \in L^p(\R/\Z)$ for $1 \leq p < \infty$, defining a bounded operator on $L^p(\R/\Z)$ if $1 < p < \infty$, and verifying a weak-type-$(1,1)$ estimate when $p=1$. In particular, $\mathcal{H}: L^2(\R/\Z) \to L^2(\R/\Z)$ can be alternatively defined via the Fourier coefficients
\begin{equation}\label{20210301_09:25}
\widehat{\mathcal{H}(f)}(k) = - i\, {\rm sgn}(k)\,\widehat{f}(k).
\end{equation}
Although we use the same notation for the Fourier transforms and Hilbert transforms on $\R$ and $\R/\Z$, it will be clear from the context which one we are referring to. We consider below some sharp inequalities for the Hilbert transform. Classical works in this theme include the ones of Pichorides \cite{Pi}, in which he finds the operator norm $\|\mathcal{H}\|_{L^p \to L^p}$ for $1 < p< \infty$ (see also \cite{Graf} for a simplified proof), and of Davis \cite{D}, in which he finds the weak-type-$(1,1)$ operator norm (such works consider both the situation in the real line and in the periodic setting).

\smallskip

Throughout the paper we let $\mathcal{A}$ be the following class of real-valued functions:
\begin{equation*}
\mathcal{A} = \left\{
\begin{array}{l}
F: \R \to \R  \ \   {\rm even, continuous\  and\  non\!-\!negative};\\
{\rm supp}(F) \subseteq [-\tfrac{1}{2},\tfrac{1}{2}] ; \\
\widehat{F} \in L^1(\R).
\end{array}
\right.
\end{equation*}
For each $F \in \mathcal{A}$ we define its periodization $f_F:\R/\Z \to \R$ by 
\begin{equation*}
f_F(\theta) := \sum_{k \in \Z} F(\theta + k).
\end{equation*}
One can verify that $f_F \in L^1(\R/\Z)$ and that $\widehat{f_F}(k) = \widehat{F}(k)$ for all $k \in \Z$. Moreover, in this situation, by a classical result of Plancherel and P\'{o}lya  
(see \cite{PP} or \cite[eq.~(3.1)]{V}),  for any $\delta > 0$ we have 
$$\sum_{k \in \Z} \big|\widehat{F}(\delta k)\big| \ll_{\delta} \big\| \widehat{F}\big\|_{L^1(\R)}.$$
In particular, for $F \in \mathcal{A}$, both $\mathcal{H}(F)$ defined by \eqref{20210301_09:24} and $\mathcal{H}(f_F)$ defined by \eqref{20210301_09:25} via Fourier inversion are bounded and continuous functions. We consider the following optimization problem involving the $L^{\infty}$-norms of these Hilbert transforms.

\smallskip

\noindent {\it Extremal Problem 1 {\rm (EP1)}.} 
With notations as above, find the infimum:
\begin{align}\label{20210222_08:46}
{\bf C} := \inf_{0 \neq F \in \mathcal{A}} \frac{\max\big\{\|\mathcal{H}(F)\|_{L^{\infty}(\R)}\,,\,\|\mathcal{H}(f_F)\|_{L^{\infty}(\R/\Z)}\big\}  }{\|F\|_{L^{1}(\R)}}. 
\end{align}

\smallskip

This problem is the main theme of study in this paper. Without necessarily knowing the precise value of the constant ${\bf C}$, our first main result gives a non-obvious theoretical connection between this optimization problem, purely in analysis, and the angular discrepancy $\mathcal{D}(P)$ of a polynomial $P$.

\begin{theorem}\label{Thm1}
Let ${\bf C}$ be given by \eqref{20210222_08:46}. If $P$ is a monic polynomial of degree $N$ with $P(0)\ne 0$, then 
$$\mathcal{D}(P) \le  \frac{4\sqrt{\bf C}}{\sqrt{\pi}}  \, \sqrt{ N \,h(P)}.$$
\end{theorem}
We prove this result in Sections \ref{Sound_Sec} and \ref{Max_HT_Sec}. From the observations leading to \eqref{20210302_21:30} and Theorem \ref{Thm1} we automatically have a lower bound coming from the number theory side:
$$ {\bf C} \geq  \frac{\pi}{16}\left(\frac{4\pi}{3 \sqrt{3} \, L(2, \chi_3)}\right) = 0.6077\ldots.$$
In Theorem 2, we go much further in our understanding of this problem. Before stating this result, we set up a second optimization problem, somewhat related to the first one. Let $\mathcal{A}^*$ be the following class of real-valued functions (slightly larger than $\mathcal{A}$):
\begin{equation*}
\mathcal{A^*} = \left\{
\begin{array}{l}
F \in L^1(\R) ,\ \   F \geq 0;\\
{\rm supp}(F) \subseteq [-\tfrac{1}{2},\tfrac{1}{2}].
\end{array}
\right.
\end{equation*}
Consider the following problem:
\smallskip

\noindent {\it Extremal Problem 2 {\rm (EP2)}.} 
With notations as above, find the infimum:
\begin{align}\label{20210315_08:46}
{\bf C^*} := \inf_{0 \neq F \in \mathcal{A^*}} \frac{\|\mathcal{H}(F)\|_{L^{\infty}(\R)} }{\|F\|_{L^{1}(\R)}}. 
\end{align}

\smallskip

Since $\mathcal{A} \subseteq \mathcal{A^*}$ and $\|\mathcal{H}(F)\|_{L^{\infty}(\R)} \le \max\big\{\|\mathcal{H}(F)\|_{L^{\infty}(\R)}\,,\,\|\mathcal{H}(f_F)\|_{L^{\infty}(\R/\Z)}\big\} $, it is clear from the definitions of (EP1) and (EP2) that 
${\bf C^*}  \leq {\bf C}.$
Our second main result establishes a complete solution for both of these extremal problems at once.

\begin{theorem}\label{Thm2}
For ${\bf C}$ given by \eqref{20210222_08:46} and ${\bf C^*}$ given by \eqref{20210315_08:46}, we have
$${\bf C^*} = {\bf C} =1.$$
Moreover, there are no extremal functions $F \in \mathcal{A}$ for the problem ${\rm (EP1)}$, and the unique \textup{(}modulo multiplication by a positive constant\textup{)} extremal function for the problem ${\rm (EP2)}$ is 
\begin{equation}\label{20210315_14:00}
\ \ \ \ \ \ \ \ \ \ \ \ \ \ \qquad \qquad \frak{F}(x) := \frac{2}{\pi}\log\left( \frac{1 + \sqrt{1 - 4x^2}}{2|x|}\right) \ \ \ \ \ \big({\rm for}\ -\tfrac12 \leq x \leq \tfrac12\big).
\end{equation}
\end{theorem}
The proof of this theorem is given in Section \ref{Magic}. The proof relies on the serendipitous existence of two {\it magic functions}:~the even function given in \eqref{20210315_14:00} and an odd function given in \eqref{20210316_14:36}. For a different perspective on lower bounds for the Hilbert transform over intervals (mostly in $L^2$), see \cite{APS}.

\smallskip

Theorem \ref{Cor2} now follows directly from Theorems \ref{Thm1} and \ref{Thm2}.

\section{Soundararajan's proof revisited}\label{Sound_Sec}

We now prepare for the proof of Theorem  \ref{Thm1}. At first, we closely follow Soundararajan's strategy of proof for the inequality  \eqref{20210227_14:47} in \cite{S}, which we briefly review for the convenience of the reader. At a certain stage of the argument (discussed in \S \ref{Ste4} below), we make a crucial change of direction that leads to our optimization problem in analysis. This is discussed in full detail in the next section, where we complete the proof of Theorem \ref{Thm1}.

\subsection{Schur's observation} \label{Sec_Schur} First note that we can assume without loss of generality that the zeros of the polynomial are all in the unit circle, an observation due to Schur \cite{Schur}. In fact, letting $P(z) = \prod_{j=1}^N \big( z-\rho_j \,e^{2\pi i\theta_j}\big)$ as above, we may define $Q(z) = \prod_{j=1}^N \big( z- e^{2\pi i\theta_j}\big)$ and observe, for $|z| =1$, that
$$\left| \frac{z}{\sqrt{\rho_j}} - \sqrt{\rho_j}\,e^{2\pi i \theta_j} \right|^2 \geq \big|z - e^{2\pi i \theta_j}\big|^2.$$
By multiplying over $j$, we find that $|P(z)|/\sqrt{|a_0|} \geq |Q(z)|$ for $|z| =1$, and therefore $h(P) \geq h(Q)$. Hence, from now on we assume that $\rho_j =1$ for $j = 1, 2, \ldots, N$. 

\subsection{Smoothed sums and $h(P)$} \label{Sec_Smooth} If we define $\psi(\theta) = \log|2 \sin(\pi \theta)|$, then its Fourier coefficients are given by $\widehat{\psi}(0) = 0$ and $\widehat{\psi}(k) = - \tfrac{1}{2|k|}$ for $k \neq 0$ (e.g.~\cite[\S 1.441, eq.~2]{GR}). Hence, for $P(z) = \prod_{j=1}^N \big( z-e^{2\pi i\theta_j}\big)$ and $k \in \Z \setminus \{0\}$, we have
\begin{align}\label{20210301_16:55}
\begin{split}
\int_{\R/\Z} e^{2 \pi i k \theta} & \log \big| P\big(e^{2\pi i \theta}\big) \big|\,\d\theta = \sum_{j =1}^N \int_{\R/\Z}e^{2 \pi i k \theta}  \log \big| e^{2\pi i \theta} - e^{2\pi i \theta_j}\big|\,\d\theta  \\
& = \sum_{j =1}^N e^{2 \pi i k \theta_j} \int_{\R/\Z} e^{2 \pi i k \alpha}  \log \big| e^{2\pi i \alpha} - 1\big|\,\d\alpha \\
& = \sum_{j =1}^N e^{2 \pi i k \theta_j}  \int_{\R/\Z} \psi(\alpha) \, e^{2 \pi i k \alpha} \,\d\alpha = - \frac{1}{2|k|}\sum_{j =1}^N e^{2 \pi i k \theta_j}.
\end{split}
\end{align}
Identity \eqref{20210301_16:55} is essentially contained in \cite[Lemma 2]{S}. Let $g: \R/\Z \to \C$ be a continuous and integrable function such that $\{k \, \widehat{g}(k)\}_{k \in \Z}$ is absolutely summable, and set 
$$\mathcal{G} := \max_{\theta} \left|\sum_{k \neq 0} 2 |k| \, \widehat{g}(k)\,  e^{2 \pi i k \theta}\right|.$$
By expanding $g$ into its Fourier series, and using \eqref{20210301_16:55}, we get
\begin{align}\label{20210301_17:46}
& \left| \sum_{j=1}^N g(\theta_j)  - N \int_{\R/\Z} g(\theta)\,\d\theta \right|  =  \left|\sum_{k \neq 0} \widehat{g}(k)  \sum_{j=1}^N  e^{2 \pi i k \theta_j}\right|   =  \left|- \int_{\R/\Z}  \log \big| P\big(e^{2\pi i \theta}\big) \big| \left(\sum_{k \neq 0} 2 |k| \widehat{g}(k) e^{2 \pi i k \theta}\right) \d\theta\right| \nonumber \\
& \ \ \ \ \leq \mathcal{G}\int_{\R/\Z}  \Big|\log \big| P\big(e^{2\pi i \theta}\big) \big| \Big| \, \d\theta = \mathcal{G} \left( \int_{\R/\Z} 2\log^+\!\big| P\big(e^{2\pi i \theta}\big) \big|\,\d\theta - \int_{\R/\Z} \log\big| P\big(e^{2\pi i \theta}\big) \big|\,\d\theta\right) \\
& \ \ \ \ = 2 \,\mathcal{G}\, h(P). \nonumber
\end{align}
In the last passage above, note the use of Jensen's formula in the identity
$$\int_{\R/\Z} \log\big| P\big(e^{2\pi i \theta}\big) \big|\,\d\theta = 0.$$
Inequality \eqref{20210301_17:46} is the content of \cite[Proposition 1]{S}.

\subsection{Majorizing the characteristic function of an interval} Having established the preliminaries in \S \ref{Sec_Schur} and \S \ref{Sec_Smooth} above, we now move on to the proof itself. First observe that if we can prove the upper bound
\begin{align}\label{20210302_08:13}
N(I; P) - |I|\,N \leq C\sqrt{N \,h(P)}\,,
\end{align}
for a certain universal constant $C$ and all intervals $I \subset \R/\Z$, we may use the identity
$$N(I; P) - |I|\, N  = |I^c|\, N - N(I^c; P),$$ 
where $I^c$ denotes the complementary interval to $I$, to obtain the corresponding lower bound. Therefore, it suffices to obtain the upper bound \eqref{20210302_08:13}.

\smallskip

Let $0 \neq F \in \mathcal{A}$, normalized so that $\|F\|_{L^1(\R)} = \int_{\R} F(x)\,\dx = 1$.  For each $0 < \delta \leq 1$, let 
$$F_{\delta}(x) := \tfrac{1}{\delta} F\big( \tfrac{x}{\delta}\big)$$ 
so that ${\rm supp}(F_{\delta}) \subset \big[-\tfrac{\delta}{2}, \tfrac{\delta}{2}\big]$. We let 
$$f_\delta(\theta) := \sum_{k \in \Z} F_{\delta}(\theta + k)$$
be the periodization of $F_{\delta}$. Note that $\int_{\R/\Z} f_\delta(\theta) \,\d\theta = 1$ and, more generally, that 
\[
\widehat{f_{\delta}}(k) = \widehat{F_{\delta}}(k) = \widehat{F}(\delta k)
\] 
for all $k \in \Z$. For each interval $I \subset \R/\Z$, let $I_{\delta}$ be the interval obtained by widening $I$ on either side by $\delta/2$; if $|I| + \delta \geq 1$, then we just consider $I_{\delta}$ to be all of $\R/\Z$. Let $\chi_{I_\delta}$ be the characteristic function of the interval $I_{\delta}$ and let $g_\delta$ be the convolution of $\chi_{I_\delta}$ and $f_{\delta}$, that is
\begin{align}\label{20210221_17:10}
g_\delta(\theta) =  \int_{\R/\Z} \chi_{I_\delta}(\alpha) \, f_{\delta}(\theta - \alpha)\,\d\alpha.
\end{align}
Note that $g_\delta$ is a continuous and non-negative function that majorizes the characteristic function of the original interval $I$. We then write
\begin{align}\label{20210221_17:09}
N(I; P) - |I|\, N  \leq \sum_{j=1}^N g_\delta(\theta_j) - |I|\,N =\left(\sum_{j=1}^N g_\delta(\theta_j)  - N \int_{\R/\Z} g_\delta(\theta)\,\d\theta \!\right) \!+ N \!\left(\int_{\R/\Z}  g_\delta(\theta)\,\d\theta -  |I|\!\right).
\end{align}

Our goal now is to bound the two terms appearing on the right-hand side of \eqref{20210221_17:09}. For the second term, we use the definition \eqref{20210221_17:10} and Fubini's theorem to get
\begin{align}\label{20210222_08:52}
0 \leq N \left(\int_{\R/\Z}  g_\delta(\theta)\,\d\theta -  |I|\right) =N \left(  \int_{\R/\Z} \chi_{I_\delta}(\alpha) \,\d\alpha -  |I|\right) \leq N\Big((|I| + \delta) - |I|\Big) = N\delta.
\end{align}
Now, if $I_\delta = [\alpha, \beta]$, for all $k \in \Z \setminus\{0\}$ we have
\begin{equation}\label{20210301_17:54}
\widehat{\chi_{I_\delta}}(k) = \frac{e^{-2 \pi i k \alpha} - e^{-2 \pi i k \beta}}{2\pi i k}.
\end{equation}
Recall that $\widehat{g_\delta}(k) = \widehat{\chi_{I_\delta}}(k)\,\widehat{f_{\delta}}(k)$ for all $k \in \Z$, hence the sequence $\{k\widehat{g_\delta}(k)\}_{k \in \Z}$ is absolutely summable. Letting 
$$\mathcal{G}_{\delta} := \max_{\theta} \left|\sum_{k \neq 0} 2 |k| \widehat{g_\delta}(k) e^{2 \pi i k \theta}\right|\,,$$
we have seen in \eqref{20210301_17:46} that the first term on the right-hand side of \eqref{20210221_17:09} satisfies
\begin{align}\label{20210301_18:24}
\left| \sum_{j=1}^N g_{\delta}(\theta_j)  - N \int_{\R/\Z} g_{\delta}(\theta)\,\d\theta \right| \leq 2 \,\mathcal{G}_{\delta}\, h(P).
\end{align}

\subsection{Understanding the cancellation} \label{Ste4} We now need to bound the quantity $\mathcal{G}_{\delta}$ and this is where we diverge from Soundararajan's original proof \cite{S}. From \eqref{20210301_17:54} we have
\begin{align}\label{20210302_08:21}
\sum_{k \neq 0} 2 |k| \widehat{g_\delta}(k) \,e^{2 \pi i k \theta} = \frac{1}{\pi} \sum_{k \neq 0} - i \, {\rm sgn}(k)\widehat{f_{\delta}}(k) \left(e^{2 \pi i k (\theta - \alpha)} - e^{2 \pi i k (\theta - \beta)}\right),
\end{align}
and hence 
\begin{align*}
\left| \sum_{k \neq 0} 2 |k| \widehat{g_\delta}(k) \,e^{2 \pi i k \theta}\right|  & \leq  \frac{1}{\pi} \left| \sum_{k \neq 0} - i \, {\rm sgn}(k)\widehat{f_{\delta}}(k)\, e^{2 \pi i k (\theta - \alpha)}\right|  + \frac{1}{\pi} \left| \sum_{k \neq 0} - i \, {\rm sgn}(k)\widehat{f_{\delta}}(k) \,e^{2 \pi i k (\theta - \beta)}\right|\\
& = \frac{1}{\pi} \big| \mathcal{H}(f_{\delta})(\theta - \alpha)\big| + \frac{1}{\pi} \big| \mathcal{H}(f_{\delta})(\theta - \beta)\big|.
\end{align*}
This plainly yields
\begin{equation}\label{20210301_18:21}
\mathcal{G}_{\delta} \leq \frac{2}{\pi} \|\mathcal{H}(f_{\delta})\|_{L^{\infty}(\R/\Z)}.
\end{equation}
Equality is actually attained if one considers the maximum over all intervals $[\alpha,\beta]$, so there is no loss in this use of the triangle inequality.

\smallskip

{\sc Remark.} In the corresponding step in \cite{S}, Soundararajan is working in the restricted subclass of $\mathcal{A}$ for which $\widehat{F} \geq 0$, and at the end he chooses $F$ to be a triangular graph. He couples the terms $k$ and $-k$ in \eqref{20210302_08:21} and uses the triangle inequality, further moving the absolute values inside the sum, to get 
$$\mathcal{G}_{\delta} \leq \frac{4}{\pi} \max_{\theta} \sum_{k \geq 1}  \widehat{f_\delta}(k) \,|\sin(2 \pi  k \theta)|.$$ 
This particular extra step of moving the absolute values inside disregards some cancellation in the sum. This is precisely the point where our analysis diverges from \cite{S}.

\smallskip

We now state a relation that is fundamental for our purposes, which essentially says that the supremum over this one-parameter family (for $0 < \delta \leq 1$) of $L^{\infty}$-norms of Hilbert transforms in \eqref{20210301_18:21}, when properly normalized, occurs at one of the endpoints $\delta = 0^+$ or $\delta = 1$.

\begin{proposition}\label{Prop_key}
Let $F \in \mathcal{A}$ and $0 < \delta \leq 1$. With notations as above, we have
\begin{align*}
\sup_{0 < \delta \leq 1} \delta \,  \|\mathcal{H}(f_{\delta})\|_{L^{\infty}(\R/\Z)} = \max\!\left\{\|\mathcal{H}(F)\|_{L^{\infty}(\R)}\,,\,\|\mathcal{H}(f_F)\|_{L^{\infty}(\R/\Z)}\right\}.
\end{align*}
\end{proposition}
We postpone the proof of this result until the next section.

\subsection{Conclusion} Assume for a moment that we have established Proposition \ref{Prop_key}. Let us simplify the notation by writing
$$\mathcal{C}(F) := \max\!\left\{\|\mathcal{H}(F)\|_{L^{\infty}(\R)}\,,\,\|\mathcal{H}(f_F)\|_{L^{\infty}(\R/\Z)}\right\}.$$
It then follows from \eqref{20210301_18:21} and Proposition \ref{Prop_key} that 
\begin{equation}\label{20210301_18:23}
\mathcal{G}_{\delta} \leq \frac{2}{\pi \delta} \,\mathcal{C}(F),
\end{equation}
and from \eqref{20210221_17:09},  \eqref{20210222_08:52}, \eqref{20210301_18:24}, and \eqref{20210301_18:23} we get
$$N(I; P) - |I|\, N  \leq N \delta + \frac{4}{\pi \delta} \,\mathcal{C}(F) \, h(P).$$
The choice of 
\begin{equation}\label{20210302_10:22}
\delta = \sqrt{\frac{4\,\mathcal{C}(F) \, h(P)}{\pi N}}
\end{equation}
minimizes the right-hand side of the expression above and leads to the bound  
$$N(I; P) - |I|\, N  \leq  \frac{4\sqrt{\mathcal{C}(F)}}{\sqrt{\pi}}  \,\sqrt{N \,h(P)}.$$
Note that this is independent of the interval $I$. Minimizing over $F \in \mathcal{A}$ we arrive at the desired conclusion
$$\mathcal{D}(P) \leq   \frac{4\sqrt{\bf C}}{\sqrt{\pi}}  \sqrt{N \,h(P)}.$$
Therefore, Theorem \ref{Thm1} follows from Proposition \ref{Prop_key}.
\smallskip

{\sc Remark.} From the fact that $\log^+\!xy \leq \log^+\!x \,+\, \log^+\!y$ for any $x,y >0$, if $P(z) = \prod_{j=1}^N \big( z-e^{2\pi i\theta_j}\big)$ we get
\begin{align*}
h(P)  =  \int_{0}^{1} \log^+ \big|P\big(e^{2\pi i \theta}\big)\big| \,\d\theta \leq N \int_{0}^{1} \log^+ \big|e^{2\pi i \theta} - 1\big| \,\d\theta = N \, \frac{ 3 \sqrt{3} \, L(2, \chi_3)}{4\pi}\,,
\end{align*}
as remarked in \eqref{20210302_10:20}. Hence, the choice of $\delta$ in \eqref{20210302_10:22} indeed falls in the interval $(0, 1]$ if 
\begin{equation}\label{20210302_10:39}
\mathcal{C}(F) \leq \frac{\pi^2}{3 \sqrt{3} \, L(2, \chi_3)} = 2.43107\ldots.
\end{equation}
In Section \ref{Sec4_Numerics}, we observe that there are functions $F \in \mathcal{A}$ that verify this bound. For example,  the triangle function $F_{\blacktriangle}(x) = 2 \, \max\big\{1 - 2|x|, 0\big\}$ has $\mathcal{C}(F_{\blacktriangle})=1.12219\ldots$. Hence, without loss of generality, we may assume that from the start we are working under the threshold \eqref{20210302_10:39}.

\section{Maxima of Hilbert transforms}\label{Max_HT_Sec}
The purpose of this section is to prove the key Proposition \ref{Prop_key}, hence concluding the proof of Theorem \ref{Thm1}. Recall that we have been using the definition of the Hilbert transforms via the multipliers \eqref{20210301_09:24} and  \eqref{20210301_09:25} and Fourier inversion (hence all Hilbert transforms here are bounded and continuous functions). In this section, the alternative representations of the Hilbert transforms as singular integrals will be particularly useful. Throughout this section we continue to assume that $0 \neq F \in \mathcal{A}$ is normalized so that $\|F\|_{L^1(\R)} = \int_{\R} F(x)\,\dx = 1$, and for each $0 < \delta \leq 1$ we let $F_{\delta}(x) := \tfrac{1}{\delta} F\big( \tfrac{x}{\delta}\big)$ and $f_\delta(\theta) := \sum_{k \in \Z} F_{\delta}(\theta + k)$. For $x \in \R$, let 
$$\|x\|: = \min\{|x - n| \, : \, n  \in \Z\}$$ 
be the distance of $x$ to the nearest integer. 

\subsection{Hilbert transforms as singular integrals} For each $0 < \delta \leq 1$, since $\mathcal{H}(f_{\delta})$ is an odd and continuous function in $\R/\Z$, we have $\mathcal{H}(f_{\delta})(0) = \mathcal{H}(f_{\delta})(\pm\tfrac12) = 0$. We start by establishing the following useful relation between the periodic Hilbert transforms $\mathcal{H}(f_{\delta})$ and the Hilbert transform $\mathcal{H}(F)$.
\begin{lemma}\label{Lem4}
Let $0 < \delta \leq 1$ and $-\frac12 < \theta < \frac12$. Then
\begin{align}\label{20210304_10:59}
\delta \,\mathcal{H}(f_{\delta})(\theta)  = \mathcal{H}(F)\left(\tfrac{\theta}{\delta}\right) +  \frac{\delta}{\pi} \sum_{k \geq 1} \int_{0}^{\frac{\delta}{2}} f_{\delta}(\alpha)\frac{4\theta(\theta^2 - \alpha^2 - k^2)}{\big((\theta - \alpha)^2 - k^2\big) \big((\theta + \alpha)^2 - k^2\big)}\,\d\alpha.
\end{align}
\end{lemma}
\begin{proof} Let $\Gamma \subset \R$ be the set of full measure (i.e.~$\R \setminus \Gamma$ has measure zero) such that for every $x \in \Gamma$ the limit 
\begin{equation}\label{20210304_16:35}
  \lim_{\varepsilon \to 0} \frac{1}{\pi} \int_{\varepsilon \leq |t|} F(x - t) \,\frac{1}{t}\,\d t 
 \end{equation}
exists and is equal to $\mathcal{H}(F)(x)$. Similarly, for a fixed $0 < \delta \leq 1$, let $\Gamma_{\delta} \subset \R/\Z$ be the set of full measure such that for every $\theta \in \Gamma_{\delta} $ the limit
\begin{equation}\label{20210304_10:44}
\lim_{\varepsilon \to 0} \int_{\varepsilon \leq |\alpha| \leq \frac12} f_{\delta}(\theta - \alpha) \,\cot(\pi \alpha)\,\d \alpha 
\end{equation}
exists and is equal to $\mathcal{H}(f_{\delta})(\theta)$.

\smallskip

Recall that, for $\|\alpha\| \geq \varepsilon>0$, we have the absolutely convergent expansion (e.g.~\cite[\S 1.421 eq.~3]{GR})
\begin{equation}\label{20210302_13:48}
\cot(\pi \alpha) = \frac{1}{\pi} \left( \frac{1}{\alpha} + \sum_{k \geq 1} \frac{2\alpha}{\alpha^2 - k^2}\right).
\end{equation}
Assume that $\theta \in \Gamma_{\delta}$ and $\frac{\theta}{\delta} \in \Gamma$. Let $\varepsilon$ be small and write $X_{\varepsilon} = \big\{ \alpha \in \big[-\frac{\delta}{2}, \frac{\delta}{2}] \, :\, \|\theta - \alpha\| \geq \varepsilon\big\}$ and $Y_{\varepsilon} = \big\{ \beta \in \big[-\frac{1}{2}, \frac{1}{2}] \, :\, \|\theta - \delta\beta\| \geq \varepsilon\big\}$. Using \eqref{20210302_13:48}, and with a change of variables $\alpha = \delta \beta$, we note that 
\begin{align*}
\delta\int_{\varepsilon \leq |\alpha| \leq \frac12} & f_{\delta}(\theta - \alpha) \,\cot(\pi \alpha)\,\d \alpha = \delta \int_{X_{\varepsilon}} f_{\delta}(\alpha) \,\cot(\pi (\theta - \alpha))\,\d \alpha\\
& = \frac{\delta}{\pi} \int_{X_{\varepsilon}}  \frac{f_{\delta}(\alpha)}{(\theta - \alpha)}\,\d \alpha \ + \frac{\delta}{\pi} \sum_{k \geq 1} \int_{X_{\varepsilon}} f_{\delta}(\alpha)\frac{2(\theta - \alpha)}{(\theta - \alpha)^2 - k^2}\,\d\alpha\\
& = \frac{1}{\pi} \int_{Y_{\varepsilon}}  \frac{F(\beta)}{\big(\frac{\theta}{\delta} - \beta\big)}\,\d \beta  + \frac{\delta}{\pi} \sum_{k \geq 1} \int_{X_{\varepsilon}} f_{\delta}(\alpha)\frac{2(\theta - \alpha)}{(\theta - \alpha)^2 - k^2}\,\d\alpha.
\end{align*}
Passing to the limit as $\varepsilon \to 0$, and using the fact that $f_{\delta}$ is even (to combine $\alpha$ and $-\alpha$ in the integral below), we get
\begin{align}\label{20210304_10:12}
\begin{split}
\delta \,\mathcal{H}(f_{\delta})(\theta) & = \mathcal{H}(F)\left(\tfrac{\theta}{\delta}\right) +  \frac{\delta}{\pi} \sum_{k \geq 1} \int_{-\frac{\delta}{2}}^{\frac{\delta}{2}} f_{\delta}(\alpha)\frac{2(\theta - \alpha)}{(\theta - \alpha)^2 - k^2}\,\d\alpha\\
& = \mathcal{H}(F)\left(\tfrac{\theta}{\delta}\right) +  \frac{\delta}{\pi} \sum_{k \geq 1} \int_{0}^{\frac{\delta}{2}} f_{\delta}(\alpha)\frac{4\theta(\theta^2 - \alpha^2 - k^2)}{\big((\theta - \alpha)^2 - k^2\big) \big((\theta + \alpha)^2 - k^2\big)}\,\d\alpha.
\end{split}
\end{align}
In principle, \eqref{20210304_10:12} holds for $\theta$ in the set of full measure $(-\tfrac12, \tfrac12) \cap \Gamma_\delta \cap \delta \Gamma$. Since the functions in \eqref{20210304_10:12} are continuous functions of $\theta \in (-\tfrac12, \tfrac12)$ we conclude that the identity is valid for all $\theta$ in this range.
\end{proof}

\subsection{Proof of Proposition \ref{Prop_key}} We start by observing that, for $0 < \delta \leq 1$ and $-\tfrac12 \leq \theta \leq -\tfrac{\delta}{2}$, we have
\begin{equation}\label{20210304_11:35}
\mathcal{H}(f_{\delta})(\theta) \leq 0.
\end{equation}
In fact, if $\delta < 1$ and $\theta \in (-\tfrac12, -\tfrac{\delta}{2}) \cap \Gamma_{\delta}$, using that $f_{\delta}$ is even, non-negative and supported in $[-\tfrac{\delta}{2}, \tfrac{\delta}{2}]$ along with the singular integral representation \eqref{20210304_10:44}, we get 
\begin{align*}
\mathcal{H}(f_{\delta})(\theta) & = \lim_{\varepsilon \to 0} \int_{X_{\varepsilon}} f_{\delta}(\alpha) \,\cot(\pi (\theta - \alpha))\,\d \alpha = \int_{-\frac{\delta}{2}}^{\frac{\delta}{2}} f_{\delta}(\alpha) \,\cot(\pi (\theta - \alpha))\,\d \alpha\\
& = 2 \int_{0}^{\frac{\delta}{2}} f_{\delta}(\alpha) \,\big(\cot(\pi (\theta - \alpha)) + \cot(\pi (\theta + \alpha))\big)\,\d \alpha \leq 0.
\end{align*}
A similar argument shows that if $x \leq -\tfrac{1}{2}$ then $\mathcal{H}(F)(x) \leq 0.$

\smallskip

Since $\mathcal{H}(f_{\delta})$ is an odd and continuous function in $[-\tfrac12, \tfrac12]$ (not identically zero), its maximum in absolute value coincides with the positive maximum, and we investigate the latter. Recall that $\mathcal{H}(f_{\delta})(0) = \mathcal{H}(f_{\delta})(\pm\tfrac12) = 0$. We split our analysis into two cases. 

\subsubsection{Case 1} Assume that $0 < \theta < \frac12$ is such that $\mathcal{H}(f_{\delta})(\theta) > 0$. In this case, the sum on the right-side of \eqref{20210304_10:59} is clearly non-positive and it plainly follows by Lemma \ref{Lem4} that 
\begin{equation}\label{20210304_12:36}
\delta \,\mathcal{H}(f_{\delta})(\theta)  \leq \mathcal{H}(F)\left(\tfrac{\theta}{\delta}\right).
\end{equation}

\subsubsection{Case 2} Assume that $-\tfrac12 < \theta < 0$ is such that $\mathcal{H}(f_{\delta})(\theta) > 0$. As observed in \eqref{20210304_11:35}, we must have $-\frac{\delta}{2} < \theta < 0$ in this situation. Using Lemma \ref{Lem4}, letting $\theta' = \frac{\theta}{\delta}$ (hence $-\frac12 < \theta' < 0$) and changing variables $\alpha = \delta \beta$ in the integral, we rewrite \eqref{20210304_10:59} as
\begin{align}\label{20210304_12:27_v2}
\delta \,\mathcal{H}(f_{\delta})(\theta) & = \mathcal{H}(F)(\theta') +  \frac{1}{\pi} \sum_{k \geq 1} \int_{0}^{\frac{1}{2}} F(\beta) \frac{4\theta'\left(\theta'^2 - \beta^2 - \left(\tfrac{k}{\delta}\right)^2\right)}{\left((\theta' - \beta)^2 - \left(\tfrac{k}{\delta}\right)^2 \right) \left((\theta' + \beta)^2 - \left(\tfrac{k}{\delta}\right)^2\right)}\,\d\beta.
\end{align} 
The important observation now is that, for each $k \geq 1$, the term
$$ \frac{4\theta'\left(\theta'^2 - \beta^2 - \left(\tfrac{k}{\delta}\right)^2\right)}{\left((\theta' - \beta)^2 - \left(\tfrac{k}{\delta}\right)^2 \right) \left((\theta' + \beta)^2 - \left(\tfrac{k}{\delta}\right)^2\right)}$$
is positive and, for fixed $-\frac12 < \theta' < 0$ and $0 \leq \beta \leq \frac{1}{2}$, the function
$$h(x) :=  \frac{4\theta' \big(\theta'^2 - \beta^2 - x^2\big)}{\big((\theta' - \beta)^2 - x^2 \big) \big((\theta' + \beta)^2 - x^2\big)}$$
verifies $h'(x) < 0$ for $x \geq 1$. This is a routine calculation. 
The conclusion is that we could replace $\delta$ on each summand on the right-hand side of \eqref{20210304_12:27_v2} by its maximum value $\delta =1$ and do better, i.e. 
\begin{align}\label{20210304_12:37}
\begin{split}
\delta \,\mathcal{H}(f_{\delta})(\theta) & \leq \mathcal{H}(F)(\theta') +  \frac{1}{\pi} \sum_{k \geq 1} \int_{0}^{\frac{1}{2}} F(\beta) \frac{4\theta'\left(\theta'^2 - \beta^2 - k^2\right)}{\big((\theta' - \beta)^2 - k^2 \big) \big((\theta' + \beta)^2 - k^2\big)}\,\d\beta\\
& = \mathcal{H}(f_{1})(\theta')\,,
\end{split}
\end{align} 
where the last identity follows from another application of Lemma \ref{Lem4}.

\subsubsection{Conclusion} From \eqref{20210304_12:36} and \eqref{20210304_12:37}, we plainly arrive at the conclusion that 
\begin{align}\label{20210304_12:39}
\sup_{0 < \delta \leq 1} \delta \,  \|\mathcal{H}(f_{\delta})\|_{L^{\infty}(\R/\Z)} \leq \max\big\{\|\mathcal{H}(F)\|_{L^{\infty}(\R)}\,,\,\|\mathcal{H}(f_F)\|_{L^{\infty}(\R/\Z)}\big\}=: \mathcal{C}(F).
\end{align}
If $\mathcal{C}(F) = \|\mathcal{H}(f_F)\|_{L^{\infty}(\R/\Z)}$, then \eqref{20210304_12:39} is obviously an equality (recall that $f_F = f_1$ in this notation). On the other hand, if $\mathcal{C}(F) = \|\mathcal{H}(F)\|_{L^{\infty}(\R)}$, let $x_0 \in \R$ be such that 
$$\mathcal{C}(F)  = \|\mathcal{H}(F)\|_{L^{\infty}(\R)} = \mathcal{H}(F)(x_0)$$
(note that $\mathcal{H}(F)$ goes to zero at infinity, hence such $x_0$ indeed exists). Let $\theta(\delta) = \delta x_0$, for $\delta$ sufficiently small so that $-\tfrac12 < \theta(\delta) < \tfrac12$. We apply Lemma \ref{Lem4} once more, by changing variables $\alpha = \delta \beta$ in the integral and rewriting \eqref{20210304_10:59} in the form
\begin{align}\label{20210304_12:27}
\delta \,\mathcal{H}(f_{\delta})(\delta x_0) & = \mathcal{H}(F)(x_0) +  \frac{1}{\pi} \sum_{k \geq 1} \int_{0}^{\frac{1}{2}} F(\beta) \frac{4\delta^2x_0\big(\delta^2x_0^2 - \delta^2\beta^2 - k^2\big)}{\big(\delta^2(x_0 - \beta)^2 - k^2 \big) \big(\delta^2(x_0 + \beta)^2 - k^2\big)}\,\d\beta.
\end{align} 
An application of the dominated convergence theorem on the right-hand side of \eqref{20210304_12:27} guarantees that 
$$\lim_{\delta \to 0^+} \delta \,\mathcal{H}(f_{\delta})(\delta x_0) = \mathcal{H}(F)(x_0)\,,$$
and we have equality in \eqref{20210304_12:39} as desired. This concludes the proof of Proposition \ref{Prop_key}.

\section{A brief interlude}\label{Sec4_Numerics}

Before moving to the final section, where we present the proof of Theorem \ref{Thm2}, let us briefly make some remarks to highlight a few important elements in our discussion. Throughout this section let $f = f_F$.

\subsection{Dichotomy} \label{Dichotomy}In the definition of $\mathcal{C}(F)$ we have a maximum between two $L^{\infty}$-norms. One may wonder if one of these is always dominated by the other. Our first observation is that this is not always the case. In principle, there are examples of functions for which either $L^{\infty}$-norm can be maximal.

\smallskip

If the maximum value of $\mathcal{H}(f)(\theta)$ occurs at a certain $0 < \theta < \tfrac{1}{2}$, then 
\begin{equation}\label{20210309_15:26}
\mathcal{C}(F) = \|\mathcal{H}(F)\|_{L^{\infty}(\R)}.
\end{equation}
This follows directly from \eqref{20210304_12:36} with $\delta =1$. This is the case, in particular, if $F$ is {\it radial decreasing}. In fact, under such assumption, for a.e.  $-\tfrac12 < \theta < 0$ we have
$$\mathcal{H}(f)(\theta)  = \lim_{\varepsilon \to 0} \int_{\varepsilon \leq |\alpha| \leq \frac12} f(\theta - \alpha) \,\cot(\pi \alpha)\,\d \alpha =  \lim_{\varepsilon \to 0} \int_{\varepsilon}^{\frac12} \big( f(\theta - \alpha) - f(\theta + \alpha)\big) \,\cot(\pi \alpha)\,\d \alpha \leq 0.$$
Since $\mathcal{H}(f)$ is continuous, this inequality is valid for all $-\tfrac12 < \theta < 0$. Note that above we used the fact that $\|\theta + \alpha\| \leq \|\theta - \alpha\|$ in our range to argue that $f(\theta - \alpha) \leq f(\theta + \alpha)$. 

\smallskip

On the other hand, if the maximum value of $\mathcal{H}(F)(x)$ occurs at a certain $-\frac12 < x < 0$ (recall that we have seen that $\mathcal{H}(F)(x) \leq 0$ for $x \leq -\tfrac12$), then
\begin{equation}\label{20210309_16:24}
\mathcal{C}(F) = \|\mathcal{H}(f)\|_{L^{\infty}(\R/\Z)}.
\end{equation}
This follows from \eqref{20210304_10:59} with $\delta =1$. There are indeed functions $F$ with such behaviour, for instance the piecewise linear function, normalized so that $\int_{\R} F(x)\,\dx = 1$,
\begin{align}\label{20210316_13:43}
F(x) = 
\begin{cases}
0, \ {\rm if} \ 0 \leq |x| \leq \frac14;\\
64|x| - 16, \ {\rm if} \ \frac14 \leq |x| \leq \frac{5}{16};\\
\frac{1}{3} (32 - 64|x|), \ {\rm if} \ \frac{5}{16} \leq |x| \leq \frac{1}{2}.
\end{cases}
\end{align}
See Figure \ref{figure_outlier} for the plots of the Hilbert transforms of this example. In a certain sense, the cases for which \eqref{20210309_16:24} holds are slightly unusual, and produce large $L^{\infty}$-norms. We prove in the next section that functions $F$ such that $\mathcal{C}(F)$ is very close to the infimum ${\bf C}$ tend to like option \eqref{20210309_15:26} better. 

\begin{figure} 
\includegraphics[scale=0.5]{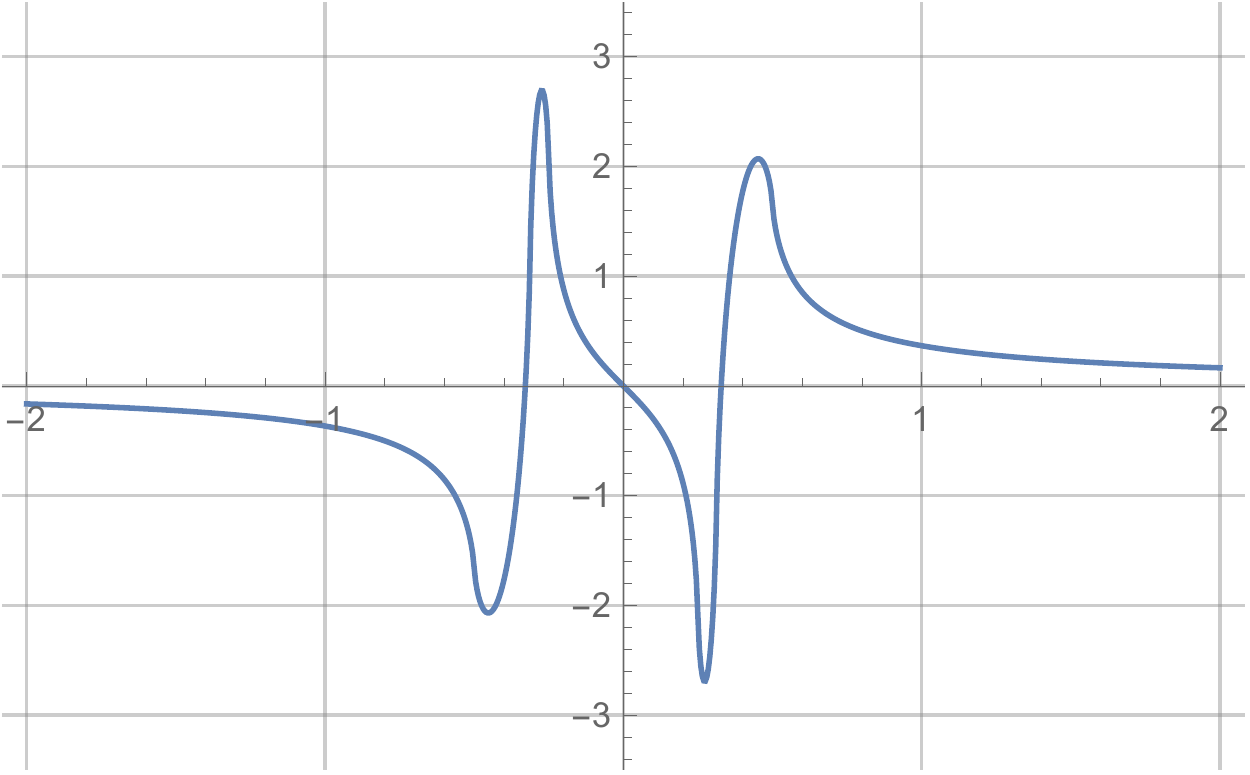}  \qquad 
\includegraphics[scale=0.5]{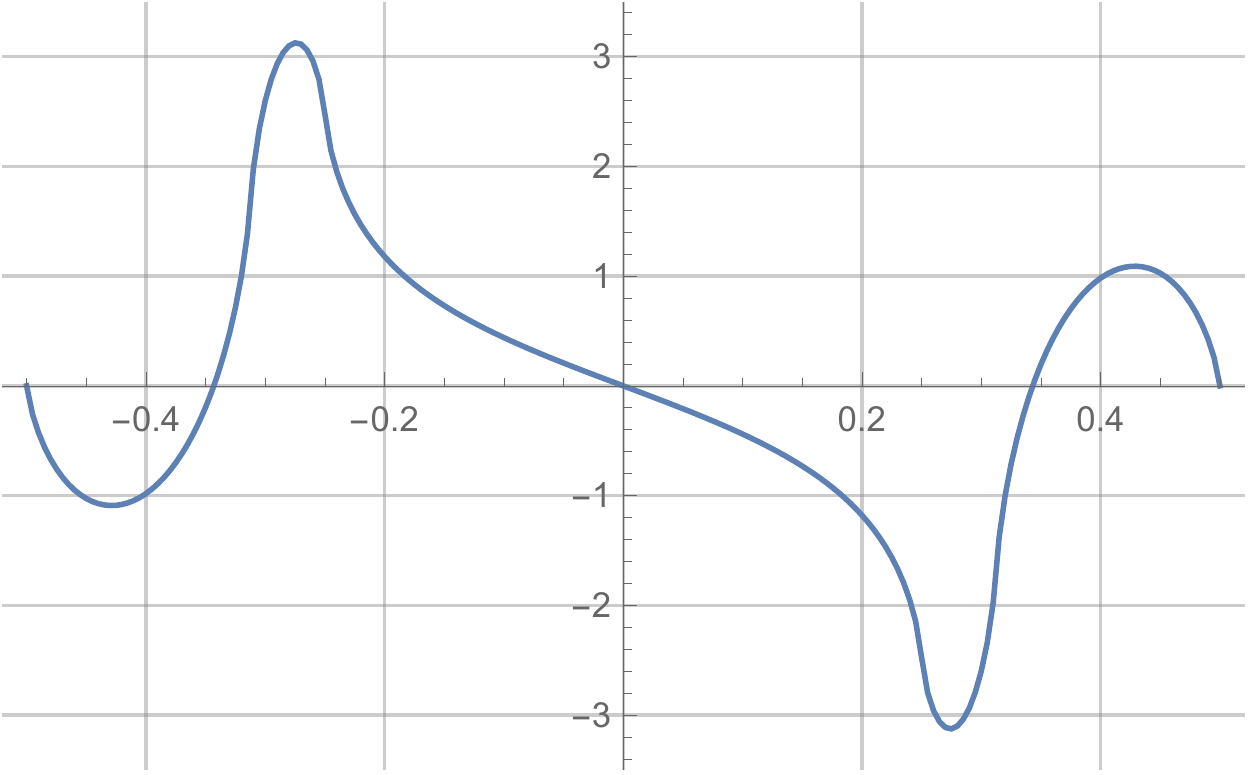}
\caption{For the function $F$ defined in \eqref{20210316_13:43}, on the left we have the graph of $\mathcal{H}(F)$ and on the right we have the graph of $\mathcal{H}(f_F)$.}
	\label{figure_outlier}
\end{figure}

\subsection{The triangle function} Consider the triangle function $F_{\blacktriangle}$ given by 
$$F_{\blacktriangle}(x) = 2 \, \max\big\{1 - 2|x|, 0\big\}.$$
Note that $\|F_{\blacktriangle}\|_{L^1(\R)} = 1$. An application of integration by parts in \eqref{20210304_16:35} shows that 
\begin{align*}
\mathcal{H}(F_{\blacktriangle})(x)  = \frac{1}{\pi} \int_0^{\frac12}F_{\blacktriangle}'(t)\,  \log \left( \frac{|x - t|}{|x+t|}\right) \,\dt  = \frac{4}{\pi}\int_0^{\frac12} \log \left( \frac{|x + t|}{|x-t|}\right) \,\dt.
\end{align*}
We seek the global maximum of $\mathcal{H}(F_{\blacktriangle})(x)$ when $x >0$. One can check that this function is decreasing if $x > \tfrac12$, simply because  $\tfrac{|x + t|}{|x-t|} < \tfrac{|y + t|}{|y-t|}$ if $x > y > \tfrac12$ for all 
$0 < t < \tfrac12$. For $0 < x < \tfrac12$, we may write
\begin{align*}
\mathcal{H}(F_{\blacktriangle})(x) = \frac{4}{\pi}\left(\int_x^{x+ \frac12} \log |y| \,\d y  - \int_{x-\tfrac12}^{x} \log |y| \,\d y \right).
\end{align*}
Hence, by the fundamental theorem of calculus, we have
\begin{align*}
\mathcal{H}(F_{\blacktriangle})'(x)= \log|x + \tfrac12| -2 \log |x| +  \log |x-\tfrac12|\,,
\end{align*}
and for $0 < x < \tfrac12$ we find that $\mathcal{H}(F_{\blacktriangle})'(x) = 0$ if and only if
$$\frac{\big(x+\tfrac12\big)\big(\tfrac12 - x\big)}{x^2} = 1\,,$$
which yields $x = 1/(2\sqrt{2})$. This is the global maximum and by \eqref{20210309_15:26} we get 
\begin{align*}
 \|\mathcal{H}(F_{\blacktriangle})\|_{L^{\infty}(\R)} = \mathcal{H}(F_{\blacktriangle})\big(\tfrac{1}{2\sqrt{2}}\big) = \frac{4}{\pi}\left(\int_{\frac{1}{2\sqrt{2}}}^{\frac{1}{2\sqrt{2}}+ \frac12} \log y \,\d y  - \int_{\frac{1}{2\sqrt{2}}-\tfrac12}^{\frac{1}{2\sqrt{2}}} \log |y|  \,\d y \right)= \frac{4}{\pi} \log(1 + \sqrt{2}). 
\end{align*} 
This shows that the constant ${\bf C}$ in (EP1) satisfies
$${\bf C} \leq \mathcal{C}(F_{\blacktriangle}) = \frac{4}{\pi} \log(1 + \sqrt{2}) = 1.12219\ldots\,,$$
and, as a consequence of Theorem \ref{Thm1}, for monic polynomials $P$ of degree $N$ with $P(0)\ne0$ we deduce that
\begin{equation}\label{20210305_09:52}
\mathcal{D}(P) \le C \sqrt{ N \,h(P)} \quad \text{with} \quad C=\frac{8}{\pi}\sqrt{\log(1 + \sqrt{2})}  = 2.3906\ldots.
 \end{equation}

\smallskip

{\sc Remark.} In \cite{S}, Soundararajan works with the triangle test function $F_{\blacktriangle}$ as above, establishing a bound in \eqref{20210227_14:47} with $C = 8/\pi = 2.54\ldots$. Later, it came to our attention that, in unpublished notes\footnote{Personal communication.}, he independently arrived at the refined inequality in \eqref{20210305_09:52} by further studying the situation with this particular test function.

\section{Magic functions} \label{Magic}

In this section we prove Theorem \ref{Thm2}. Ultimately, our proof relies on the existence of two {\it magic functions}. The first one, mentioned in the statement of the theorem, is the even function, supported in $[-\frac12, \frac12]$, 
\begin{equation}\label{20210316_14:49}
\frak{F}(x) := \frac{2}{\pi}\log\left( \frac{1 + \sqrt{1 - 4x^2}}{2|x|}\right) \ \ \ \ \ \big({\rm for}\ -\tfrac12 \leq x \leq \tfrac12\big).
\end{equation}
The second one is the odd function, also supported in $[-\frac12, \frac12]$, given by
\begin{equation}\label{20210316_14:36}
\qquad \qquad  \ \ \  \frak{G}(x) :=  \frac{2x}{\sqrt{1 - 4x^2}}  \ \ \ \ \ \qquad  \big({\rm for}\ -\tfrac12 < x < \tfrac12\big).
\end{equation}

\smallskip

We first treat the extremal problem (EP1), to find the value of the sharp constant ${\bf C}$. Later, with some of the main ingredients already laid out, we discuss the details that lead to the solution of the extremal problem (EP2) and the sharp constant ${\bf C}^*$.
 
\subsection{Lower bound via duality} The map $\mathcal{H}:L^2(\R) \to L^2(\R)$ is an isometry and verifies $\mathcal{H}^2 = - I$ (therefore the inverse of $\mathcal{H}$ is $-\mathcal{H}$). Hence, whenever $F_1, F_2 \in L^2(\R)$ we have 
\begin{equation}\label{20210310_07:54}
\int_{\R} \mathcal{H}(F_1)(x) \, \overline{F_2(x)}\,\dx = - \int_{\R} F_1(x)\, \overline{\mathcal{H}(F_2)(x)}\,\dx. 
\end{equation}
Since $\mathcal{H}: L^p(\R) \to L^p(\R)$ is a bounded operator for $1 < p < \infty$, identity \eqref{20210310_07:54} extends to the situation where $F_1 \in L^p(\R)$ and $F_2 \in L^{p'}(\R)$, where $\tfrac{1}{p} + \tfrac{1}{p'} = 1$ and $1 < p,p' < \infty$. The odd function $\frak{G}$ belongs to $L^p(\R)$ for $1 \leq p < 2$ but not to $L^2(\R)$. It verifies
\begin{equation}\label{20210309_18:13}
\|\frak{G}\|_{L^1(\R)} = \int_{\R} |\frak{G}(x)|\,\dx = 1.
\end{equation}
The Hilbert transform of $\frak{G}$ can be explicitly computed and is given by 
\begin{equation}\label{20210309_18:14}
\mathcal{H}(\frak{G})(x) = 
\begin{cases} 
-1, \ {\rm if}\  |x| < \tfrac12;\\
-1 + \frac{2|x|}{\sqrt{4x^2 -1}}, \ {\rm if}\  |x| > \tfrac12.
\end{cases}
\end{equation}
We refer the reader to \cite[p.~248, eq.~(25)]{Ba} for this computation\footnote{Letting $L(x)$ be the function on the left-hand side  of \cite[p. 248, eq.~(25)]{Ba} with $a = \frac12$, we have $\frak{G}(x) = \tfrac12\big(L(x) - L(-x)\big)$. Note also that the Hilbert transform in \cite{Ba} is defined with a multiplying factor of $-1$.}. We shall see in a moment that the fact that $\frak{G}$ has $L^1(\mathbb{R})$-norm equal to $1$ and that its Hilbert transform is constant (equal to $-1$) in the interval $[-\tfrac12,\tfrac12]$ is precisely what makes it magical. The graphs of $\frak{G}$ and $\mathcal{H}(\frak{G})$ are plotted in Figure \ref{figure_G}.

\smallskip

Take any $0 \neq F \in \mathcal{A}$, normalized so that $\|F\|_{L^1(\R)} = \int_{\R} F(x)\,\dx = 1$. Note that $F \in L^p(\R)$ for all $1 \leq p \leq \infty$. Using \eqref{20210310_07:54} with $F_1 = F$ and $F_2 = \frak{G}$, together with \eqref{20210309_18:13}, \eqref{20210309_18:14}, and the fact that ${\rm supp}(F) \subset [-\tfrac12, \tfrac12]$, we get the following relation
\begin{align}\label{20210310_08:21}
\begin{split}
\|\mathcal{H}(F)\|_{L^{\infty}(\R)} & = \|\mathcal{H}(F)\|_{L^{\infty}(\R)} \int_{\R} |\frak{G}(x)|\,\dx  \geq  \int_{\R} \mathcal{H}(F)(x) \, \overline{\frak{G}(x)}\,\dx \\
& = - \int_{\R} F(x)\, \overline{\mathcal{H}(\frak{G})(x)}\,\dx = \int_{-\frac12}^{\frac12} F(x)\,\dx = 1.
\end{split}
\end{align}
Since \eqref{20210310_08:21} holds for any such normalized $F \in \mathcal{A}$, we plainly get the lower bound
$${\bf C} \geq 1.$$

In addition, once we establish in the next subsection that ${\bf C}$ is actually equal to $1$, relation \eqref{20210310_08:21} also tells us that there are no extremizers for the problem (EP1) in the class $\mathcal{A}$. In fact, equality in \eqref{20210310_08:21} could only be attained if 
\begin{equation*}
\mathcal{H}(F)(x) = {\rm sgn}(x)\,\|\mathcal{H}(F)\|_{L^{\infty}(\R)} = {\rm sgn}(x)
\end{equation*}
for a.e.~$-\tfrac12 < x < \tfrac12$, which cannot occur since $\mathcal{H}(F)$ is odd and continuous when $F \in \mathcal{A}$.

\begin{figure} 
\includegraphics[scale=0.5]{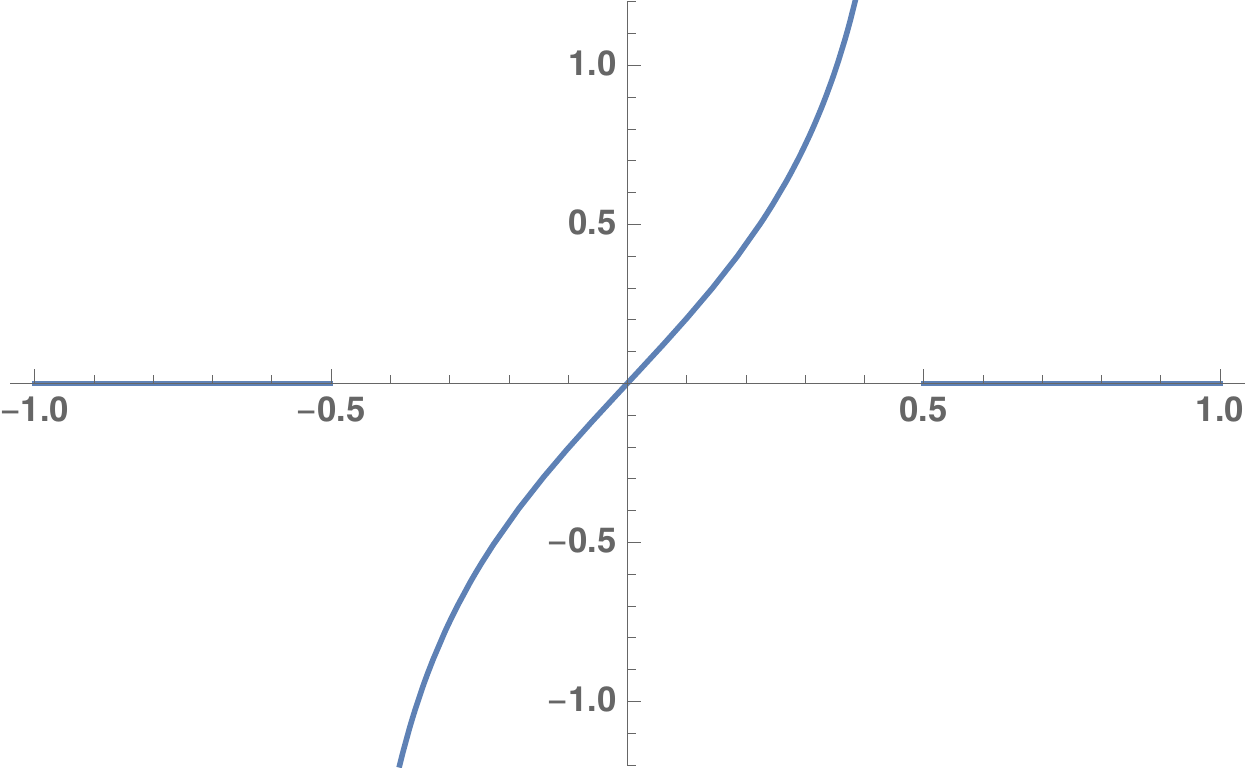}  \qquad 
\includegraphics[scale=0.5]{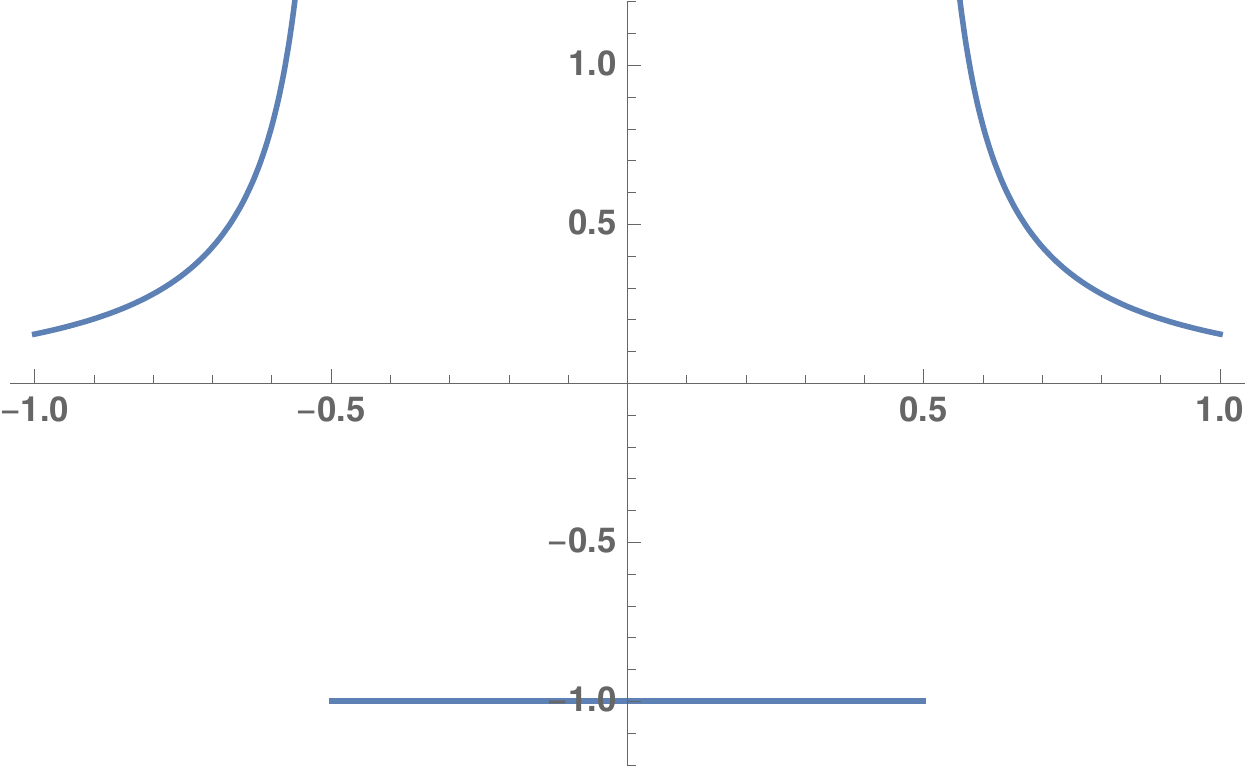}
\caption{On the left, the graph of the magic function $\frak{G}$. On the right, the graph of the Hilbert transform $\mathcal{H}(\frak{G})$.}
	\label{figure_G}
\end{figure}

\subsection{A rogue extremal function} We now turn our attention to the function $\frak{F}$ defined in \eqref{20210316_14:49}. Observe first that $\frak{F} \geq0$ in $[-\tfrac12, \tfrac12]$, $\frak{F}$ is continuous and radial decreasing on $\R\setminus\{0\}$ (with a logarithmic singularity at the origin), and $\frak{F}$ is smooth on $\R \setminus \{0, \pm\tfrac12\}$. Moreover, $\frak{F} \in L^p(\R)$ for $1 \leq p < \infty$, and we note that 
\begin{equation}\label{20210310_12:36}
\|\frak{F}\|_{L^1(\R)} = \int_{\R} \frak{F}(x)\,\dx = 1.
\end{equation}

\subsubsection{The Hilbert transform of $\frak{F}$} The Hilbert transform $\mathcal{H}(\frak{F})$ is an odd function. For almost every $x \in \R$ it is given by its singular integral representation. We may carefully apply integration by parts (excluding the singularities and then passing to the limit) to get, for a.e.~$x >0$,
\begin{align*}
\mathcal{H}(\frak{F})(x) & = {\rm p.v.} \, \frac{1}{\pi} \int_{\R} \frak{F}(x-t)\,\frac{1}{t}\,\dt = \frac{1}{\pi} \int_{\R} \frak{F}'(x-t)\,\log|t|\,\dt = \frac{1}{\pi} \int_{-\frac12}^{\frac12}\frak{F}'(t) \log|x-t|\,\dt\\
&= - \frac{1}{\pi} \int_{0}^{\frac12}\frak{F}'(t) \log\left(\frac{|x+t|}{|x-t|}\right)\,\dt \\
& = \frac{2}{\pi^2}\int_{0}^{\frac12}\frac{1}{t\,\sqrt{1 - 4t^2}} \log\left(\frac{|x+t|}{|x-t|}\right)\,\dt.
\end{align*}
This last integral can be evaluated explicitly, see \cite[\S 4.297 eqs.~8 and 10]{GR}, yielding
\begin{align}\label{20210310_12:37}
\mathcal{H}(\frak{F})(x) = 
\begin{cases} 
{\rm sgn}(x), \ {\rm if}\  |x| \leq \tfrac12;\\
\frac{2}{\pi} \arcsin\left(\frac{1}{2x}\right), \ {\rm if}\  |x| > \tfrac12.
\end{cases}
\end{align}
From \eqref{20210310_12:37} we see that $\|\mathcal{H}(\frak{F})\|_{L^{\infty}(\R)} = 1$, and given that $\frak{F}$ has the correct normalization \eqref{20210310_12:36}, it is essentially an extremizer for our problem. We say `essentially'  because $\frak{F}$ does not exactly belong to our class $\mathcal{A}$, but it is almost there (hence the {\it rogue} in the title of this subsection). The graphs of $\frak{F}$ and $\mathcal{H}(\frak{F})$ are plotted in Figure \ref{figure_F}.

\subsubsection{Approximating the rogue extremal function} We need to make a small correction to $\frak{F}$ via a standard approximation argument. Let $\varphi \in C^{\infty}_c(\R)$ be a non-negative radial decreasing function supported in $[-\tfrac12,\tfrac12]$, with $\int_{\R} \varphi(x)\,\dx = 1$ and Fourier transform $\widehat{\varphi}$ also non-negative. To construct such a function, we can just take $\varphi = \psi*\psi$, where $\psi$ is a smooth, radial decreasing, non-negative function supported in $[-\tfrac14, \tfrac14]$. Recall that the convolution of two radial decreasing functions is still radial decreasing (for a beautiful proof of this fact we refer the reader to \cite[p.~171]{Be}). For $\varepsilon >0$ small, let 
\[
\varphi_{\varepsilon}(x) := \tfrac{1}{\varepsilon} \, \varphi\!\left( \tfrac{x}{\varepsilon}\right) \quad {\rm and} \quad \frak{F}_{1-\varepsilon}(x) := \tfrac{1}{1-\varepsilon} \, \frak{F}\!\left( \tfrac{x}{1-\varepsilon}\right),
\]
and define
$$F^{\varepsilon} :=\frak{F}_{1-\varepsilon} * \varphi_{\varepsilon}.$$
Observe that $F^{\varepsilon}$ is a smooth, radial decreasing and non-negative function, supported in $[-\tfrac12,\tfrac12]$ with $\int_{\R} F^{\varepsilon}(x)\,\dx = 1$. Moreover, $\widehat{F^{\varepsilon}}(t) = \widehat{\frak{F}}\big((1-\varepsilon)t\big)\, \widehat{\varphi}(\varepsilon t) \in L^1(\R)$ (recall that $\widehat{\frak{F}}$ is bounded since $\frak{F} \in L^1(\R)$). Hence $F^{\varepsilon} \in \mathcal{A}$. 
At the level of the Hilbert transform, we have
\begin{align*}
\mathcal{H}(F^{\varepsilon})(x) = \big(\mathcal{H}(\frak{F}_{1-\varepsilon})* \varphi_{\varepsilon}\big) (x)= \left(\tfrac{1}{1- \varepsilon} \mathcal{H}(\frak{F})\left(\tfrac{^\centerdot}{1 - \varepsilon}\right)*\varphi_{\varepsilon}\right)\!(x),
\end{align*}
and we see from \eqref{20210310_12:37} that 
$$\|\mathcal{H}(F^{\varepsilon})\|_{L^{\infty}(\R)} = \frac{1}{1 - \varepsilon}.$$
As we have argued in \S\ref{Dichotomy}, since $F^{\varepsilon}$ is radial decreasing we do have $\mathcal{C}(F^{\varepsilon}) = \|\mathcal{H}(F^{\varepsilon})\|_{L^{\infty}(\R)}$. Sending $\varepsilon \to 0$ we conclude that 
$${\bf C} = 1.$$

\begin{figure}  \label{Magic_F}
\includegraphics[scale=0.5]{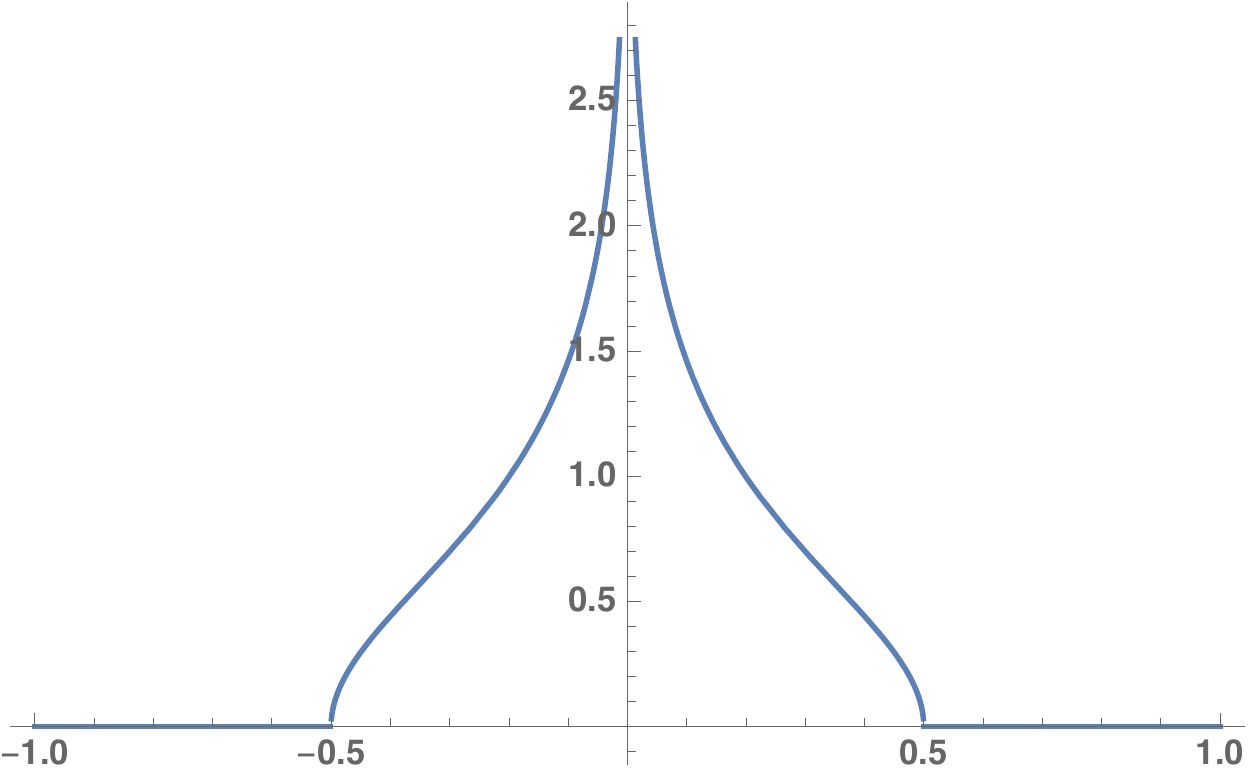}  \qquad 
\includegraphics[scale=0.5]{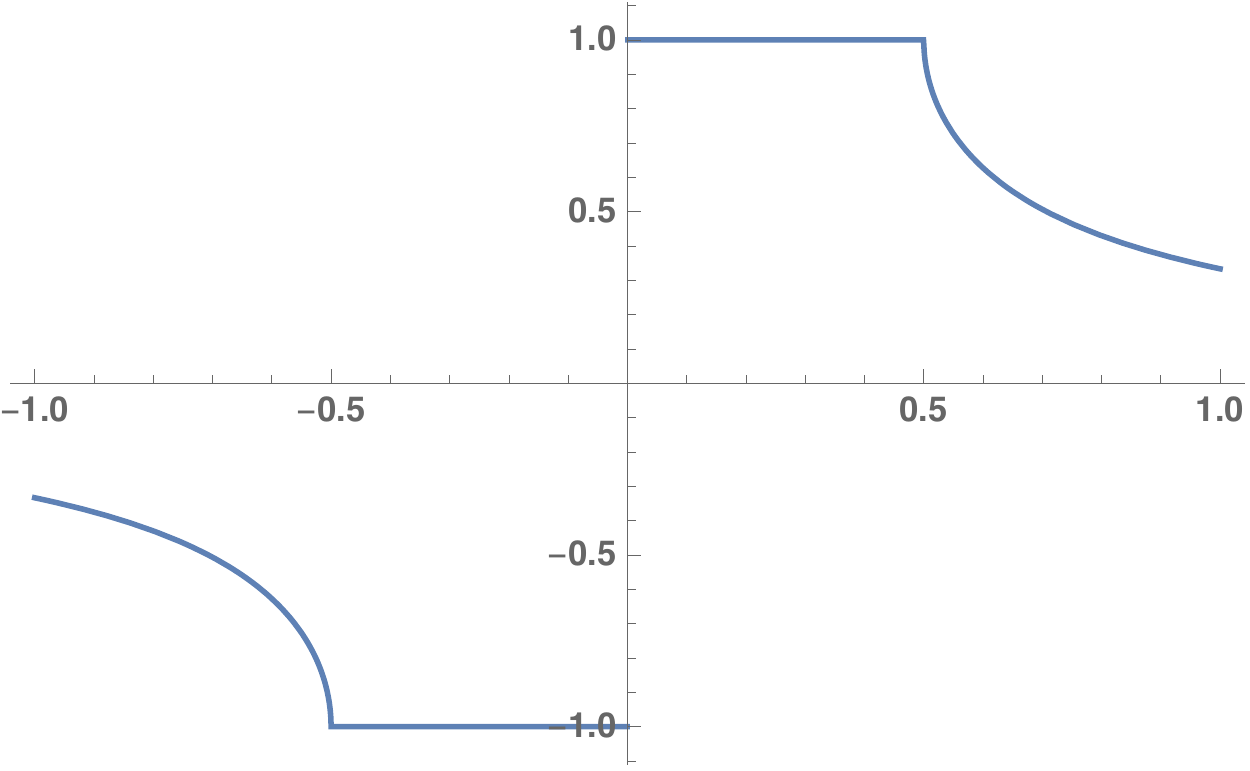}
\caption{On the left, the graph of the magic function $\frak{F}$. On the right, the graph of the Hilbert transform $\mathcal{H}(\frak{F})$.}
	\label{figure_F}
\end{figure}

\subsection{The extremal problem (EP2)} Note that $\frak{F} \in \mathcal{A}^*$ and that ${\bf C}^* \leq {\bf C} = 1$. We now verify the lower bound. Let $0 \neq F \in \mathcal{A}^*$ be a given function, normalized so that $\|F\|_{L^1(\R)} = \int_{\R} F(x)\,\dx = 1$. We may assume without loss of generality that $\|\mathcal{H}(F)\|_{L^{\infty}(\R)} < \infty$. Since $\mathcal{H}(F)(x) = 1/(\pi x) + O\big(1/|x|^2\big)$ for $|x|$ large, we find that $\mathcal{H}(F) \in L^p(\R)$ for any $1 < p \leq \infty$. This implies that $F$ must have been in $L^p(\R)$, for any $1 \leq p < \infty$, from the start.

\smallskip 

This last claim deserves a brief justification. An argument of Calder\'{o}n and Capri \cite[Lemma 4]{CC}\footnote{This lemma is stated for the situation when the singular integral operator belongs to $L^1$, but the proof works for $L^p$ ($1 < p < \infty$) as well. One simply applies Minkowski's inequality for integrals to arrive at eq.~(17) with $1 < p < \infty$.} shows that whenever $F \in L^1(\R)$ and $\mathcal{H}(F) \in L^p(\R)$, for some $1 < p < \infty$, and $\Psi$ is a continuous function of compact support, for a.e.~$x\in \mathbb{R}$ we have
\begin{equation}\label{20210323_14:08}
\mathcal{H}(F* \Psi)(x) = \big(\mathcal{H}(F) * \Psi\big)(x). 
\end{equation}
Letting $\varphi$ be a smooth function of compact support, with $\int_{\R} \varphi(x)\,\dx = 1$, and setting $\varphi_{\varepsilon}(x) := \tfrac{1}{\varepsilon} \varphi\left( \tfrac{x}{\varepsilon}\right)$ as usual,
identity \eqref{20210323_14:08} holds with $\Psi$ replaced by $\varphi_{\varepsilon}$. Since $F*\varphi_{\varepsilon}$ and $\mathcal{H}(F)*\varphi_{\varepsilon}$ belong to $L^p(\R)$ we may apply the Hilbert transform on both sides of \eqref{20210323_14:08}, using the fact that $\mathcal{H}^2 = -I$ on $L^p(\R)$, to arrive at 
\begin{equation}\label{20210323_14:09}
-(F*\varphi_{\varepsilon})(x) = \mathcal{H}\big(\mathcal{H}(F) * \varphi_{\varepsilon}\big)(x) 
\end{equation}
for a.e.~$x\in \mathbb{R}$. Letting $\varepsilon \to 0$, since $\mathcal{H}(F) * \varphi_{\varepsilon} \to \mathcal{H}(F)$ in $L^p(\R)$ and the Hilbert transform is bounded on $L^p(\R)$, the right-hand side of \eqref{20210323_14:09} converges to $\mathcal{H}\big(\mathcal{H}(F)\big)$ in $L^p(\R)$. The left-hand side of \eqref{20210323_14:09} converges to $-F$ a.e. The conclusion is that we must indeed have $-F = \mathcal{H}\big(\mathcal{H}(F)\big)$, and therefore $F \in L^p(\R)$ as well. An alternative way to argue when $F$ has compact support and $\int_{\R} F(x)\,\dx = 1$, is by observing that $F - \chi_{[-\frac12, \frac12]}$ belongs to the Hardy space $H^1(\R) = \{G \in L^1(\R) \ ; \mathcal{H}(G) \in L^1(\R)\}$. This is a Banach space with norm $\|G\|_{H^1(\R)}:=  \|G\|_{L^1(\R)} + \|\mathcal{H}(G)\|_{L^1(\R)}$ (see e.g.~\cite[Theorem 6.7.4]{Graf_book}) in which $\mathcal{H}$ is an isometry with $\mathcal{H}^2 = -I$.
\smallskip

Having gone through these considerations, the application of \eqref{20210310_08:21} is justified and we arrive at the conclusion that ${\bf C}^* \geq 1$, and hence ${\bf C}^* = 1$. Let us now discuss the uniqueness of the extremizer. Equality happens in \eqref{20210310_08:21} if and only if
\begin{equation*}
\mathcal{H}(F)(x) = {\rm sgn}(x)\,\|\mathcal{H}(F)\|_{L^{\infty}(\R)} = {\rm sgn}(x)
\end{equation*}
for a.e.~$-\tfrac12 < x < \tfrac12$. This implies that 
\begin{equation}\label{20210318_15:45}
\mathcal{H}(\frak{F} - F)(x) = 0
\end{equation}
for a.e.~$-\tfrac12 < x < \tfrac12$. We are now in position to invoke a suitable uniqueness result, first established in a classical paper by Tricomi \cite{T}, and revisited recently by Coifman and Steinerberger \cite[Theorem 1]{CS}.
\begin{lemma}[cf.~\cite{T} and \cite{CS}]\label{Lem_uniqueness}
Let $G$ be a real-valued function such that ${\rm supp}(G) \subset [-\tfrac12, \tfrac12\big]$ and $G(x)(1- 4x^2)^{1/4} \in L^2\big(-\tfrac12, \tfrac12\big)$. If $\mathcal{H}(G) \equiv 0$ on $\big(-\tfrac12, \tfrac12\big)$ then, for some $c \in \R$, we have
$$\qquad \qquad  \qquad \qquad G(x) = \frac{c}{\sqrt{1 - 4x^2}}   \qquad \ \  \big({\rm for}\ -\tfrac12 < x < \tfrac12\big).$$
\end{lemma}
From \eqref{20210318_15:45} and Lemma \ref{Lem_uniqueness} we arrive at 
$$\frak{F}(x) - F(x) =  \frac{c}{\sqrt{1 - 4x^2}}$$
for a.e.~$-\tfrac12 < x < \tfrac12$. Since $\int_{\R} \frak{F}(x)\,\dx = \int_{\R} F(x)\,\dx = 1$, we conclude that $c=0$ and $F = \frak{F}$, as proposed.

\section*{Acknowledgements}
This project started at the workshop {\it Arithmetic statistics, discrete restriction, and Fourier analysis} at the American Institute of Mathematics (AIM) in 2021. We thank Theresa Anderson, Frank Thorne, and Trevor Wooley for the organization of the workshop, and the staff of AIM for providing a superb scientific atmosphere. We are thankful to William Beckner, Tiago Picon, Ruiwen Shu, Mateus Sousa, and the referee for enlightening remarks. We also thank Kannan Soundararajan for sharing some of his unpublished notes on the theme. EC acknowledges support from FAPERJ - Brazil, AF was supported by NSF DMS-2101769 and the NSF Postdoctoral Fellowship DMS-1703695, AM was supported by NSF DMS-1854398 FRG, MBM was supported by NSF DMS-2101912 and a Simons Foundation Collaboration Grant for Mathematicians, and CT-B was supported by NSF DMS-1902193 and NSF DMS-1854398 FRG.

\end{document}